\documentclass[11pt,]{article}

\usepackage{amssymb,amsmath}
\usepackage{amsthm,enumerate}
\usepackage{graphicx}

\numberwithin{equation}{section}

\newtheorem{theorem}{Theorem}
\newtheorem{corollary}{Corollary}
\newtheorem{lemma}{Lemma}

\theoremstyle{definition}
\newtheorem*{example}{Example}
\newtheorem*{remark}{Remark}

\newtheorem*{definition}{Definition}

\long\def\symbolfootnote[#1]#2{\begingroup%
\def\thefootnote{\fnsymbol{footnote}}\footnote[#1]{#2}\endgroup}

\numberwithin{equation}{section}

\newcommand{\be}{\begin{equation}}
\newcommand{\ee}{\end{equation}}

\begin{document}

\def\rc#1{\frac{1}{#1}}
\def\NN{\mathbb{N}}

\title{Some Properties and Applications of Non-trivial Divisor Functions}
\date{\today}

\author{S. L. Hill, M. N. Huxley, M. C. Lettington, K. M. Schmidt}
\maketitle

\begin{abstract}
The $j$th divisor function $d_j$, which counts the ordered factorisations of a
positive integer into $j$ positive integer factors, is a very well-known
arithmetic function;
in particular, $d_2(n)$ gives the
number of divisors of $n$.
However, the $j$th non-trivial divisor function $c_j$, which counts the ordered proper
factorisations of a positive integer into $j$ factors, each of which is
greater than or equal to 2, is rather less well-studied. We also consider associated divisor functions $c_j^{(r)}$,
whose definition is motivated by the sum-over divisors recurrence for $d_j$.
After reviewing properties of $d_j$, we study analogous properties of $c_j$ and $c_j^{(r)}$,
specifically regarding their Dirichlet series and generating functions,
as well as representations in terms of binomial coefficient sums and hypergeometric series.
We also express their ratios as binomial coefficient
sums and hypergeometric series, and find explicit Dirichlet series and Euler
products in some cases.
As an illustrative application of the non-trivial and associated divisor
functions, we show how they can be used to count principal
reversible square matrices of the type considered by Ollerenshaw and Br\'ee, and hence sum-and-distance systems of integers.
\end{abstract}

\medskip
\noindent
{\bf Acknowledgement}\quad
Sally Hill's research was supported by Engineering and Physical Sciences Research Council DTP grant EP/L504749/1.

\section{Introduction}
\label{s1}
The $j$th divisor function $d_j$, which counts the ordered factorisations of a
positive integer into $j$ positive integer factors, is a very well-known
arithmetic function.
In particular, $d_2(n)$ --- sometimes called the divisor function --- counts
the number of ordered pairs of positive integers whose product is $n$, and
therefore, considering only the first factor in each pair, also counts the
number of divisors of $n$
%
%
(see papers 8 and 15 of \cite{rSR} and p.~10 of \cite{rPGLD}).
The divisor function lies at the heart of a number of open number theoretical
problems, e.g. the {\it additive divisor problem\/} of finding the asymptotic
of
\begin{equation}\label{adp}
\sum_{n\le x} d_j(n)\,d_j(n+h)
\end{equation}
for large $x$, which is
notoriously difficult if $j\ge 3$,
see e.g.~\cite{IWu}, \cite{ABR}, and, for $j=3$, \cite{Ivic}.

In the present paper, we consider the rather less well-studied $j$th
{\it non-trivial divisor function\/} $c_j$, which counts the ordered proper
factorisations of a positive integer into $j$ factors, each of which is
greater than or equal to 2.
While $d_j(n)$, for given $n$, is obviously monotone increasing in $j$, since
factors of 1 can be freely introduced, $c_j(n)$ will shrink back to 0 as
$j$ gets too large, and indeed
%
$c_j(n) = 0$ if $n < 2^j$.

Additionally we define the {\it associated divisor function\/} $c_j^{(r)}$, for $r\in\mathbb{N}_0$, by
\[c_j^{(0)}= c_j, \qquad  c_j^{(r)}(n) = \sum_{m|n}c^{(r-1)}_j(m)\quad (n, r\in\mathbb{N}). \]
This definition is motivated by the sum-over divisors recurrence for $d_j$, see Lemma \ref{djrec}.

The paper is organised as follows. After reviewing properties of $d_j$ in
Section 2, we proceed to study analogous properties of $c_j$ in Section 3,
specifically regarding its associated Dirichlet series and generating function,
and its representation in terms of binomial coefficient sums and
hypergeometric series.
A major complication in comparison to $d_j$ arises from the fact that $c_j$
is not multiplicative. We also provide formulae expressing $c_j$ in terms of
$d_j$ and vice versa.
In Section 4, we study the Dirichlet series and generating function for the
associated divisor functions $c_j^{(r)}$.
Noting general inequalities between the three types of divisor function in
Section 5, we observe how their ratios can be expressed as binomial coefficient
sums and hypergeometric series, and find explicit Dirichlet series and Euler
products for some of these.
As an illustrative application of the non-trivial and associated divisor
functions, we show in Section 6 how they can be used to count principal
reversible squares \cite{rOB} and sum-and-distance systems of integers.

Throughout the paper we use the notations $\mathbb{N} = \{1, 2, 3, \dots\}$,
$\mathbb{N}_0~=~\mathbb{N}~\cup~\{0\}$. For $n$ having prime factorisation
$n~=~ p_1^{a_1} p_2^{a_2} \cdots p_t^{a_t}$, we also use the symbol
$\Omega(n) = \sum\limits_{k=1}^t a_k$ .
\section{The multiplicative arithmetic function $d_j(n)$}
\begin{lemma}\label{djrec}
	Let $j,n \in \mathbb{N} $, $j \ge 2$. Then the $j$th divisor function $d_j(n)$ satisfies the sum-over-divisors recurrence relation
	\[d_j(n)=\sum_{m|n}d_{j-1}(m) .\]
\end{lemma}
	\begin{proof}
In the general
ordered factorisation in $j$ factors, $n = m_1 m_2\cdots m_j$,  $m= \frac{n}{m_j}=m_1 \cdots m_{j-1} $ can be any divisor of $n$, and $m_1\cdots m_{j-1} $ is any ordered factorisation of $m$. Hence there are $\sum_{m|n}d_{j-1}(m) $ distinct ordered factorisations of $n$ in $j$ factors.
	\end{proof}

\begin{lemma}\label{bcrt}
Let $p_1, \dots, p_t$ be distinct primes, $t\in\NN$.
Then, for any $j\in\NN$,
\begin{equation}
\label{bcrtf}
d_j(p_1^{a_1} p_2^{a_2} \dots p_t^{a_t}) = \prod_{k=1}^{t}  {a_k +j -1 \choose a_k} \quad (a_1, \dots, a_t\in\mathbb{N}_0).
\end{equation}
\end{lemma}
\begin{proof}[by induction on $j$]
For $j=1$ the formula is trivial.
Suppose $j\in\NN$ is such that (\ref{bcrtf}) holds. Note that if
$n = \prod\limits_{i=1}^k p_i^{a_i}$, then
$m | n$ if and only if
$m = \prod\limits_{i=1}^k p_i^{\tilde a_i}$
with $0 \le \tilde a_i \le a_i$ for all $i \in\{1, \dots, k\}$.
Using multiindex notation, we can write write the latter condition in the form
$0 \le \tilde a \le a$.
Hence, by Lemma \ref{djrec},
\begin{align*}
d_{j+1}(p_1^{a_1} \cdots p_k^{a_k}) &= \sum_{0 \le\tilde a \le a} d_j(p_1^{\tilde a_1} \cdots p_k^{\tilde a_k})
= \sum_{0 \le\tilde a \le a} \prod_{i=1}^k {\tilde a_i + j-1 \choose \tilde a_i}
= \prod_{i=1}^k \sum_{l=0}^{a_i} {l + j-1 \choose l}
\\
&= \prod_{i=1}^k {a_i + j \choose a_i},
\end{align*}
using combinatorial identity (1.49) of \cite{rHWG} in the last step.
\end{proof}
\begin{corollary}\label{cor3}
For any $j\in\NN$, the $j$th divisor function $d_j$ is a multiplicative arithmetic function, i.e.\ $d_j(mn)=d_j(m)d_j(n)$, whenever $(m,n)=1$.
\end{corollary}
\noindent
This follows directly from (\ref{bcrtf}), considering that $m$ and $n$ have no
prime in common, so the right-hand side can be rearranged into two products
with disjoint index sets. The arithmetic function $d_j$
is not totally multiplicative. For example,
$d_3(20)=18\neq 27=3\times 9=d_3(2)d_3(10)$.

\medskip
\noindent
In the following, we
denote by $\zeta$ the Euler-Riemann zeta function (see e.g.~\cite{rHME} or \cite{tit1}),
\[\zeta(s)= \sum^{\infty}_{n=1}\frac{1}{n^s} = \prod_{p \ {\rm prime}} \left( 1- \frac{1}{p^s} \right)^{-1}
\qquad (\mathop{\rm Re} s > 1).
\]
\begin{lemma}\label{3}
For $j\in\NN$, the divisor function $d_j$ has the Dirichlet series
	\[\sum_{n=1}^\infty \frac{d_j(n)}{n^s}=\zeta(s)^j. \]
\end{lemma}.

\begin{proof}	
By repeated application of Lemma \ref{djrec},
	\[d_j(n)=\sum_{n_1|n}\sum_{n_2|n_1} \ldots \sum_{n_{j-1} | n_{j-2}} d_1(n_{j-1})
 = \sum_{(n_1, \dots, n_j)\in\NN^j : \prod_{k=1}^j n_k = n} 1.
\]	
Hence the Dirichlet series for $d_j$ takes the form
\begin{align*}
	\sum_{n=1}^{\infty} \frac{d_j(n)}{n^s}
&= \sum_{n=1}^\infty \rc{n^s} \sum_{(n_1, \dots, n_j)\in\NN^j : \prod_{k=1}^j n_k = n} 1
=\sum_{n_1=1}^{\infty}\sum_{n_2=1}^{\infty}\ldots \sum_{n_j=1}^{\infty}\frac{1}{\left(\prod_{k=1}^j n_k\right)^s}
\\
&= \prod_{k=1}^j \sum_{n_k=1}^\infty \rc{n_k^s}
= \zeta(s)^j.
\end{align*}
\end{proof}	
\begin{corollary}\label{multidised}
	For $i,j \in \mathbb{N} $ and suitable $s \in \mathbb{C} $,		
	\[\left(\sum_{n=1}^{\infty}\frac{d_i(n)}{n^s}\right)\left(\sum_{n=1}^{\infty}\frac{d_j(n)}{n^s}\right)= \sum_{n=1} ^{\infty} \frac{d_{j+i}(n)}{n^s}.\]	
\end{corollary}
\begin{remark}
The statements of Lemma \ref{3} and Corollary \ref{multidised} extend to the divisor function $d_0$ defined as
	\[d_0(n)=\left\{\begin{array}{cc}
	1 &  \textrm{if}\ n =1   \\
	0 &  \textrm{if}\ n>1
	\end{array}\right. \qquad (n\in\mathbb{N}). \]
	There are other known Dirichlet series associated with $d_j$ (see for example \cite{rHW} Chapter XVII), such as
	\[  \frac{\zeta(s)^3}{\zeta(2s)} =\sum_{n=1} ^{\infty}\frac{d_2(n^2)}{n^s}, \qquad\textrm{and} \qquad \frac{\zeta(s)^4}{\zeta(2s)}
=\sum_{n=1}^\infty \frac{\left (d_2(n)\right )^2}{n^s},\]
both of which can be obtained from the more general calculation
(where we use the Euler product for the zeta function and Newton's inverse
binomial series in the first two steps, respectively),
\begin{align*}
\frac{\zeta(s)^{r+2}}{\zeta(2s)} &=\prod_p \frac{(1-p^{-s})(1+p^{-s})}{(1-p^{-s})(1-p^{-s})^{r+1}}
=\prod_p (1+p^{-s})\sum_{j=0}^\infty \frac{(j+r)!}{j!r!}p^{-sj}
\\
&=\prod_p \left (1+\sum_{j=1}^\infty \left (\frac{(j+r)!}{j!r!} + \frac{(j+r-1)!}{(j-1)!r!}\right )p^{-sj}\right )
\\
&=\prod_p\sum_{j=0}^\infty \frac{(2j+r)(j+r-1)!}{j!r!}p^{-sj}=\sum_{n=0}^\infty \frac{f(n)}{n^s},
\end{align*}
where
$f$ is the multiplicative arithmetic function
\[
f(p_1^{a_1}p_2^{a_2}\ldots p_t^{a_t})=\prod_{j=1}^t\frac{(2a_j+r)(a_j+r-1)!}{r!a_j!}.
\]
\end{remark}

\noindent
If we define the generating function for the divisor function $d_j$ as
\[
D_j(x) = \sum_{n=1}^\infty d_j(n)\,x^n \qquad (j\in\mathbb{N})
\]
for values of $x$ for which the power series converges, then we have the
following recursive identity which reflects the sum-over-divisors recurrence
of Lemma \ref{djrec}.
\begin{lemma}
Let $j\in\mathbb{N}$. Then for all $x\in\mathbb{C}$ for which the generating function $D_j(x)$ is defined,
\[ D_j(x) = \sum_{k=1}^\infty D_{j-1}(x^k).\]
	
\end{lemma}		
\begin{proof}
We observe that
\begin{align*}
\sum_{n=1}^\infty d_j(n) x^n &= \sum_{k\in\mathbb{N}^j} x^{k_1 \cdots k_j}
= \sum_{k_j=1}^\infty \sum_{k\in\mathbb{N}^{j-1}} x^{(k_1 \cdots k_{j-1}) k_j}
= \sum_{k_j=1}^\infty \sum_{n=1}^\infty d_{j-1}(n) (x^{k_j})^n,
\end{align*}
and the stated identity follows.
\end{proof}

\noindent
We remark that by a similar calculation,
\begin{align*}
\sum_{n=1}^{\infty} d_j(n) x^n
&=\sum_{k \in \mathbb{N}^{j-1}} \sum_{k_j=1}^{\infty}(x^{k_1\ldots k_{j-1}})^{k_j}
=\sum_{k \in \mathbb{N}^{j-1}} \frac{x^{k_1 \ldots k_{j-1}}}{1-x^{k_1 \ldots k_{j-1}}} = \sum_{n=1}^{\infty} d_{j-1}(n) \frac{x^n}{1-x^n}
\end{align*}
provided $x \neq 1$,
however, the right-hand side is then not of the form of a generating function.

\section{The non-trivial divisor function $c_j$}
The arithmetic function $c_j$ satisfies a sum-over-divisors recurrence with respect to $j$ analogous to, but subtly different from, that given for $d_j$ in Lemma \ref{djrec}.
\begin{lemma}\label{cjrec}
	Let $j,n \in \mathbb{N} $, $j \ge 2$. Then the $j$th non-trivial divisor function satisfies the sum-over-divisors recurrence relation
	\[c_j(n)=\mathop{\sum_{m|n}}_{m<n} c_{j-1}(m)=\mathop{\sum_{m|n}}_{m\notin \{1,n\}} c_{j-1}(m) .\]
\end{lemma}
	\begin{proof}
In the general ordered non-trivial factorisation of $n$ into $j$ factors, $n = m_1 \dots m_j$, the factor $m_j$ can be any non-trivial divisor of $n$,
so $m= \frac{n}{m_j}=m_1 \dots m_{j-1} $ is any proper divisor of $n$, and $m_1\ldots m_{j-1} $ is any non-trivial ordered factorisation of $m$.
Thus there are $\sum\limits_{m|n,\,m<n}c_{j-1}(m) $ distinct $j$-factor ordered non-trivial factorisations of $n$. The last identity follows as $c_j(1)=0$ for any $j \in \mathbb{N}$.
\end{proof}

\begin{lemma}\label{cjzeta}
		For $j\in\mathbb{N}$, the non-trivial divisor function $c_j$ has the Dirichlet series
	\[\sum_{n=1}^{\infty} \frac{c_j(n)}{n^s} = \left(\zeta(s)-1\right)^j. \]
\end{lemma}
	\begin{proof} We have
	\begin{align*}
	\sum_{n=1}^{\infty} \frac{c_j(n)}{n^s} &= 	\sum_{n=2^j}^{\infty} \frac{c_j(n)}{n^s}   = \sum_{n=2^j}^{\infty} \underbrace{\sum_{n_1} \sum_{n_2} \dots \sum_{n_j}}_{n_1\cdots n_j = n, \ n_i \neq 1} \frac{1}{n^s}  \\
	 & = \sum_{n_1=2}^{\infty} \sum_{n_2=2}^{\infty} \dots \sum_{n_j =2}^{\infty} \frac{1}{(n_1n_2 \cdots n_j)^s} =\left(  \sum_{n_1=2}^{\infty} \frac{1}{n_1^s}\right) \cdots \left(\sum_{n_j=2}^{\infty} \frac{1}{n_j^s}\right) \\
	 &= (\zeta(s)-1)\cdots (\zeta(s)-1) = (\zeta(s)-1)^j.
	\end{align*}
	\end{proof}

\noindent
We remark that, unlike $d_j$, $c_j$ is \emph{not} a multiplicative arithmetic function.
For example,
$(2,5)=1$, and yet \(c_2(10)=2\neq 0 \times 0 = c_2(2) c_2(5).\)

In order to study the less symmetric multiplicative properties of $c_j$, it is useful to express it in terms of its multiplicative cousin $d_j$.
When $j=2$, the non-trivial divisors for any $n$ are all the divisor except $1$ and $n$, and hence
$c_2(n)=d_2(n)-2$
if $n \ge 2$, and $c_2(1) = d_2(1) - 1$.
More generally, there is the following connection between the divisor function
and the non-trivial divisor function.
\begin{lemma}\label{dscj}
	For $j, n\in\NN$,
	\[c_j(n)= \sum^{j}_{i=0}(-1)^{j-i} {j\choose i} d_i(n).\]
\end{lemma}

\begin{proof}
By Lemma \ref{cjzeta} and Lemma \ref{3},
\begin{align*}
	\sum^{\infty}_{n=1} \frac{c_j(n)}{n^s} &= (\zeta(s)-1)^j
=\sum^{j}_{i=0}(-1)^{j-i} {j \choose i} \zeta^{j-i}(s)
	=\sum_{n=1}^{\infty} \frac{1}{n^s} \sum^j_{i=0}(-1)^{j-i} {j \choose i}  d_i(n),
\end{align*}
and the claimed identity follows by the uniqueness of the coefficients of
Dirichlet series (cf. \cite{rHR}).
\end{proof}

\noindent
The formula in Lemma \ref{dscj} extends to $j=0$ if we define $c_0(n) := d_0(n)$
$(n\in\mathbb{N})$.
With this convention, we have the following inversion formula.

\begin{lemma}
	For $j, n\in\NN$,
	\[d_j(n)= \sum_{i=0}^{j} {j \choose i} c_{i}(n) .  \]
\end{lemma}
	\begin{proof}
		Considering the Dirichlet series for the right-hand side of the above equation, we find
by Lemma \ref{cjzeta}
\begin{align*}
		 \sum_{n=1}^{\infty} \frac{1}{n^s}\sum_{i=0}^{j} {j \choose i} c_{i}(n) &= \sum_{i=0}^{j} {j \choose i}\sum_{n=1}^{\infty} \frac{1}{n^s} c_{i}(n)
 =\sum_{i=0}^{j} {j \choose i}(\zeta(s)-1)^{i}
\\
 &= (1+(\zeta(s) -1))^j = \zeta(s)^j = \sum_{n=1}^{\infty}\frac{d_j(n)}{n^s}
\end{align*}
by Lemma \ref{3},
		and the result follows by the uniqueness of coefficients of Dirichlet series (cf. \cite{rHR}).
	\end{proof}

\noindent
For the generating function of $c_j$,
\[
C_j(x) = \sum_{n=2^j}^\infty c_j(n) x^n,
\]
there is the following recursive identity.
\begin{lemma}
Let $j\in\mathbb{N}$. Then for all $x\in\mathbb{C}$
for which $C_j(x)$ is defined,
\[
 C_j(x) = \sum_{k=2}^\infty C_{j-1}(x^k).
\]
\end{lemma}
\begin{proof}
\begin{align*}
\sum_{n=2^j}^{\infty}c_j(n) x^n &= \sum_{k_1=2}^\infty \sum_{k_2=2}^\infty \dots \sum_{k_j =2}^\infty x^{k_1 k_2 \cdots k_j}
= \sum_{k_2}^\infty \left(\sum_{k_1=2}^\infty \dots \sum_{k_{j-1}=2}^\infty (x^{k_j})^{k_1\cdots k_{j-1}} \right)
\\
&= \sum_{k_j=2}^\infty \sum_{n=2^{j-1}}^\infty c_{j-1}(n) (x^{k_j})^n
\end{align*}
and the statement follows.
\end{proof}

\noindent
We remark that grouping the terms in the calculation in the preceding proof differently yields
the identity (for $x \neq 1$)
		\begin{align*}
		\sum_{n=2^j}^{\infty}c_j(n) x^n
		&= \sum_{k_1=2}^{\infty} \dots \sum_{k_{j-1}=2}^{\infty} \left(\sum_{k_j=2}^{\infty} (x^{k_1 \cdots k_{j-1}})^{k_j} \right) \\
		=& \sum_{k_1=2}^\infty \dots \sum_{k_{j-1}=2}^\infty \frac{x^{2 k_1 \cdots k_{j-1}}}{1-x^{k_1 \cdots k_{j-1}}}
		= \sum_{n=2^{j-1}}^\infty c_{j-1}(n) \frac{x^{2n}}{1-x^n}.
		\end{align*}

\noindent
We conclude this section with the derivation of a hypergeometric series for $c_j(n)$.
The generalised hypergeometric series has the form
	\[_{k}F_n (a_1,a_2, \dots, a_k; b_1,b_2,\dots, b_n;z)= \sum_{m=0}^{\infty} \frac{a_1^{\overline{m}} \ a_2^{\overline{m}}\cdots a_k^{\overline{m}}\,z^m}{b_1^{\overline{m}} \ b_2^{\overline{m}}\cdots b_n^{\overline{m}}\,m!},  \]
where $a^{\overline{m}}$, with $m\in\mathbb{N}$, is the Pochhammer symbol (rising factorial)
\[
 a^{\overline{m}} = \prod_{j=0}^{m-1}(a+j);
\]
in particular,
$1^{\overline{m}} = m!$,
$2^{\overline{m}} = (m+1)!$,
$a^{\overline{m}} = (a+m-1)!/(a-1)!$ if $a\in\mathbb{N}$ and, for negative $a$,
$a^{\overline{m}} = (-1)^m\,(-a)!/(-a-m)!$ if $-(a+m)\in\mathbb{N}_0$.
By the usual convention on empty products,
$a^{\overline{0}} =1 $.

%

\begin{theorem}\label{thyper}
Let $j\in\mathbb{N}$ and suppose
$n$ has the prime factorisation $n=p_1^{a_1}\dots p_k^{a_k}$.
Then the value of the non-trivial $j$th divisor function at $n$ has the hypergeometric form
	\[c_j(p_1^{a_1}\cdots p_k^{a_k})=(-1)^{1-j}j\,_{k+1}F_{k}(\{a_i+1\}^{k}_{i=1},(1-j);\{1\}_{i=1}^{k-1},2;1). \]
\end{theorem}	
\begin{proof}
Starting from the right-hand side expression, we find
\begin{align*}
(-1)^{1-j}j &\sum_{m=0}^{\infty}\frac{\left(\prod_{i=1}^k (a_i+1)^{\overline{m}}\right)(1-j)^{\overline{m}}}{(1^{\overline{m}})^{k-1}\,2^{\overline{m}}\,m!}
\\
&= (-1)^{1-j} \sum^{j-1}_{m=0}\frac{j\,(j-1)!\,(-1)^m}{(m+1)!\,(j-m-1)!}\prod_{i=1}^k\frac{(a_i+m)!}{a!\,m!} \\
&= \sum^{j-1}_{m=0}(-1)^{m-j+1} {j \choose m+1}\prod_{i=1}^k {a_i+m \choose m}
\\
 &= \sum^{j-1}_{m=0}(-1)^{m-j+1} {j \choose m+1}d_{m+1}(p_1^{a_1}\cdots p_k^{a_k}) =c_j(p_1^{a_1}\cdots p_k^{a_k}),
\end{align*}
	by Lemmata \ref{bcrt} and \ref{dscj}.
	\end{proof}

\noindent
In particular, for a prime power the above theorem gives
\[c_j(p^a)= (-1)^{1-j}j \,_{2}F_1(a+1,1-j;2;1)=
(-1)^{1-j}j \sum_{m=0}^{\infty}\frac{(a+1)^{\overline{m}}(1-j)^{\overline{m}}}{2^{\overline{m}}\,m!}. \]
Finally, we note the following multiplication rule for prime powers.
\begin{lemma}\label{lemd2}
Let $p$ be a prime and $j, a, b \in\mathbb{N}$.
Then
\begin{align*}
c_j(p^{a+b})
&=\sum^{j-1}_{k=0} (-1)^{k-j+1} {j \choose k+1 } d_{k+1}(p^b)\frac{(b+k+1)^{\overline{a}}}{(b+1)^{\overline{a}}} .
\end{align*}
	\end{lemma}
\begin{proof}
By Lemma \ref{bcrt},
\begin{align*}
\frac{d_j(p^{a+b})}{d_j(p^b)} &= \frac{(a+b+j-1)!\,b!\,(j-1)!}{(a+b)!\,(j-1)!\,(b+j-1)!}
= \frac{(b+j)^{\overline{a}}}{(b+1)^{\overline{a}}}.
\end{align*}
The statement now follows by combining this result with Lemma \ref{dscj}.
\end{proof}

\section{The associated divisor function $c_j^{(r)} $ }
In analogy to the sum-over-divisors recurrence relation for the divisor function $d_j$ (Lemma \ref{djrec}), we define the
$j$th associated divisor function $c_j^{(r)} $
by the following recurrence.

\begin{definition}
	Let $j\in\mathbb{N}$. Then, for all non-negative integers $r$, the {\it associated divisor function\/} $c_j^{(r)}$ is defined recursively by
	\[c_j^{(0)}(n)= c_j(n), \qquad  c_j^{(r)}(n) = \sum_{m|n}c^{(r-1)}_j(m) \qquad (n\in\mathbb{N}). \]
From the corresponding property of $c_j$, it follows that $c_j^{(r)}(n) = 0$
for all $r\in\mathbb{N}$ if $n < 2^j$.
\end{definition}

\begin{lemma}\label{cjr1}
	Let $j\in\mathbb{N}$ and $r\in\mathbb{N}_0$. Then the associated divisor function $c_j ^{(r)} $ can be expressed in terms of the divisor function $d_j$ as follows,	
	\[c_j^{(r)}(n)= \sum_{i=0}^j (-1)^{j-i} {j \choose i} d_{i+r}(n)\qquad (n\in\mathbb{N}).  \]
\end{lemma}	
	\begin{proof}[by induction on $r$]
		For $r=0 $ this is the statement of Lemma \ref{dscj}.
Now suppose that $r \in \mathbb{N}$ is such that
		\[ c_j^{(r)}= \sum_{i=0}^j (-1)^{j-i} {j \choose i} d_{i+r}\]		
		holds. Then
		by Lemma \ref{djrec},
\begin{align*}
c_{j}^{(r+1)}(n)&= \sum_{m|n}c_j^{(r)}(m)= \sum_{m|n}\sum_{i=0}^j (-1)^{j-i} {j \choose i} d_{i+r}(m)
\\
 &=\sum_{i=0}^j (-1)^{j-i} {j \choose i}\sum_{m|n} d_{i+r}(m),
\end{align*}
and the claimed statement follows by Lemma \ref{djrec}.
	\end{proof}	
\begin{lemma}\label{mixcjr}
	Let $j\in\mathbb{N}$ and $r\in\mathbb{N}_0$. Then the associated divisor function $c_j^{(r)}$ satisfies
	\[c_j^{(r)}= \sum_{i=0}^r { r \choose i} c_{j+i}. \]
\end{lemma}
\begin{proof}[by induction on $r$]
The identity is trivial for $r=0$.
	Now assume $r \in \mathbb{N}$ is such that
the above identity holds. Then, using Lemma \ref{cjrec},
\begin{align*}
 c_j^{(r+1)}(n)&= \sum_{m|n}c_j^{(r)}(n) = \sum_{m|n}\sum_{i=0}^r {  r\choose i }  c_{j+i}(m) = \sum_{i=0}^r {  r\choose i } \sum_{m|n} c_{j+i}(m)
\\
&= \sum_{i=0}^r {  r\choose i } \left(c_{j+i}(n) +  \sum_{m|n, \ m \neq 1} c_{j+i}(m)\right) =\sum_{i=0}^r {  r\choose i } \left(c_{j+i}(n) +   c_{j+i+1}(n)\right)
\\
&=\sum_{i=0}^r {r \choose i}c_{j+i}(n) + \sum_{i=1}^{r+1}{ r \choose i-1} c_{j+i}(n)
\\
&= c_j(n) + \sum_{i=1}^{r}\left({r \choose i} + {r \choose i-1 }\right)c_{j+i}(n) + c_{j+r+1}(n)
\\
&= \sum_{i=0}^{r+1} {r+1  \choose i}c_{j+i}(n) \qquad (n\in\mathbb{N}).
\end{align*}
\end{proof}

\begin{lemma}\label{mcl8}
For $j\in\mathbb{N}$ and $r\in\mathbb{N}_0$, the Dirichlet series of the associated divisor function $c_j^{(r)}$ is given by
\[\sum_{n=1}^{\infty} \frac{c_j^{(r)}(n)}{n^s} = \zeta(s)^r (\zeta(s) -1)^j.\]
\end{lemma}
\begin{proof}[by induction on $r$]
The case $r=0$ was shown in Lemma \ref{cjzeta}.
For the induction step, note that, by a well-known result on Dirichlet
convolution (see e.g.\ \cite{tit1} p.\ 4), the identity
$c_j^{(r+1)}(n) = \sum\limits_{m|n} c_j^{(r)}(m)$ $(n\in\mathbb{N})$ implies that
	\[\sum_{n=1}^{\infty}\frac{c_j^{(r+1)}(n)}{n^s} = \zeta(s)\sum_{n=1}^\infty \frac{c_j^{(r)}(n)}{n^s}.\]
\end{proof}

\noindent
The sum over divisors taken as the definition of $c_j^{(r)}$ in terms of
$c_j^{(r-1)}$ gives rise to the following recursion formula for the
generating functions
\[
C_j^{(r)}(x) = \sum_{n=2^j}^\infty c_j^{(r)}(n) x^n.
\]

\begin{lemma}
Let $j, r \in \mathbb{N}$. Then for all $x\in\mathbb{C}$ for which $C_j^{(r)}$ is defined,
\[
 C_j^{(r)}(x) = \sum_{k=1}^\infty C_j^{(r-1)}(x^k).
\]
\end{lemma}
\begin{proof}
We observe
\begin{align*}
\sum_{n=2^j}^\infty c_j^{(r)}(n) x^n &= \sum_{n=2^j}^\infty \sum_{m|n} c_j^{(r-1)}(m) x^n
 = \sum_{k=1}^\infty \sum_{m=2^j}^\infty c_j^{(r-1)}(m) (x^k)^m,
\end{align*}
and the statement follows.
\end{proof}

\noindent
Again, we note that, if $x\neq 1$, then a different arrangement of terms in the above calculation
gives the different identity
		\begin{align*}
		\sum_{n=2^j}^{\infty} c_j^{(r)}(n)x^n & = \sum_{n=2^j}^{\infty}\sum_{m |n} c_j^{(r-1)}(m)x^n
  = \sum_{m=2^j}^{\infty} \sum_{k=1}^\infty c_j^{(r-1)}(m)\,(x^m)^k \\
&= \sum_{n=2^j}^\infty c_j^{(r-1)}(n)\,\frac{x^{n}}{1-x^n}.
		\end{align*}


\noindent
To conclude this section, we prove a binomial form for the value of
$c_j^{(r)}$ at prime powers; this is somewhat analogous to Lemma \ref{bcrt}, but
note that the present function is not multiplicative.
\begin{lemma}\label{binomcd}
Let $j, a \in\mathbb{N}$, $r \in\mathbb{N}_0$, and $p$ a prime. Then
\[
c_j^{(r)}(p^a)=\binom{a+r-1}{j+r-1}.
\]
\end{lemma}
\begin{proof}
From Lemma \ref{cjr1} and Lemma \ref{bcrt} we find
\begin{align*}
c^{(r)}_j(p^a) &= \sum_{i=0}^j (-1)^{j-i} {j \choose i}d_{i+r}(p^a)= \sum_{i=0}^j (-1)^{j-i} {j \choose i} {a +i+r -1 \choose a}
\\
 &= {a + r - 1 \choose a - j} = {a + r -1 \choose j + r - 1}
\end{align*}
by combinatorial identity (3.47) of \cite{rHWG}.
\end{proof}
\section{Ratios of Divisor Functions}

\noindent
The divisor function, non-trivial divisor function and associated divisor
functions are in the following ordering relation to each other.
\begin{lemma}\label{lorder}
For any $j\in\mathbb{N}$ and $r\in\mathbb{N}$,
\[ c_j^{(r-1)}(n) \le c_j^{(r)}(n); \]
in particular,
\[ c_j(n) \le c_j^{(r)}(n) \qquad (n\in\mathbb{N}). \]
Moreover,
\[ c_j^{(r)}(n) \le d_{j+r}(n) \qquad (n\in\mathbb{N}). \]
\end{lemma}
\begin{proof}
For the first statement, it is sufficient to note that $n | n$, so
\[ c_j^{(r)}(n) = \sum_{m|n} c_j^{(r-1)}(m) \ge c_j^{(r-1)}(n) \qquad (n\in\mathbb{N}). \]
We prove the final statement by induction on $r$.
As every non-trivial factorisation is a factorisation,
it follows directly from their definitions that
$d_j(m) \ge c_j(m)$
for all $m, j\in\mathbb{N}$.
Hence
\[ d_{j+1}(n) = \sum_{m|n} d_j(m) \ge \sum_{m|n} c_j(m) = c_j^{(1)}(n), \]
which gives the case $r=1$. Now suppose $r\in\mathbb{N}$ is such that the
claimed statement holds. Then
\[ d_{j+r+1}(n) = \sum_{m|n} d_{j+r}(m) \ge \sum_{m|n} c_j^{(r)}(n) = c_j^{(r+1)}(n), \]
which is the induction step.
\end{proof}
%

\noindent
Lemma \ref{lorder} shows that, for any $r\in\mathbb{N}_0$,
the normalised divisor ratio function $c_j^{(r)}/d_{j+r}$ takes rational values between 0~and~1, with the zeros occurring exactly when $j>\Omega(n)$.
We have the following formulae for this function and the similar ratio
$c_j^{(r)}/d_r$.

\begin{theorem}\label{cjr2}
Let $j\in\mathbb{N}$ and $r\in\mathbb{N}_0$, and suppose $n\in\mathbb{N}$ has
prime factorisation
$n=p_1^{a_1} \cdots p_t^{a_t} $. Then
\begin{align}
\frac{c_j^{(r)}(n)}{d_{j+r}(n)} &=  \sum_{i=0}^j (-1)^i {j \choose i} \frac{{ j+r-1  \choose i}^t}{\prod _{k=1}^t {a_k +j +r -1 \choose i}} \nonumber \\
 &=   {}_{t+1}F_t(\{1-j-r\}_{i=1}^t,-j;\{1-a_i-j-r \}^{t}_{i=1};1).
\label{eq:thm12}
\end{align}
Also, for $r \ge 1$
\begin{align}
 \frac{c_j^{(r)}(n)}{d_{r}(n)} &= \sum _{i=0}^j (-1)^{j-i} \binom{j}{i} \frac{\prod _{k=1}^{t}
   \binom{a_k+i+r-1}{i}}{\binom{i+r-1}{i}^{t}}
 \nonumber \\
 &= (-1)^{j}\ {}_{t+1}F_{t}(\{a_k+r\}^{t}_{k=1},-j;\{r\}_{k=1}^{t};1).
\label{eq:thm13}
\end{align}
\end{theorem}
\begin{proof}
By Lemma \ref{cjr1} and Lemma \ref{bcrt}, we have that
\begin{align}
c_j^{(r)}(n) &=\sum_{i=0}^{j}(-1)^i {j \choose i} d_{j+r-i}(n)= \sum_{i=0}^{j}(-1)^i {j \choose i}\prod_{k=1}^t \frac{(a_k +j+r-1-i)!}{(j+r-i-1)!\, a_k!} \label{eq:rider} \\
 &=\sum_{i=0}^{j}(-1)^i {j \choose i}\prod_{k=1}^{t} \frac{ {j+r-1 \choose i}  {a_k+j +r -1 \choose j+r-1}  } {{a_k +j +r -1 \choose i }} \nonumber \\
 &=\left(\prod_{k=1}^{t} {a_k+j+r-1  \choose j+r-1}\right)\sum_{i=0}^{j}(-1)^i {j \choose i} \frac{{j+r-1 \choose i }^t}{\prod_{k=1}^t {a_k+j+r-1 \choose i}} \nonumber \\
 &=d_{j+r}(n) \sum_{i=0}^j (-1)^i {j \choose i} \frac{{ j+r-1  \choose i}^t}{\prod _{k=1}^t {a_k +j +r -1 \choose i}}. \nonumber
\end{align}
To obtain the hypergeometric form, we note that
%
%
%
%
%
%
\begin{align*}
\frac{c_j^{(r)}(n)}{d_{j+r}(n)} &= \sum_{i=0}^j \frac{(-j)^{\overline i}}{i!}\,\frac{\left(\frac{(j+r-1)!}{(j+r-1-i)!}\right)^t}{\prod_{k=1}^t \frac{(j+r+a_k-1)!}{(j+r+a_k-1-i)!}} \\
&= \sum_{i=0}^\infty \frac{(-j)^{\overline i}}{i!}\,\frac{\left((-1)^i (1-j-r)^{\overline i}\right)^t}{\prod_{k=1}^t (-1)^i (1-j-r-a_k)^{\overline i}} \\
&= {}_{t+1}F_t(\{1-j-r\}_{k=1}^t, -j; \{1-a_k-j-r\}_{k=1}^t; 1),
\end{align*}
establishing (\ref{eq:thm12}).
For the proof of (\ref{eq:thm13}), we rewrite (\ref{eq:rider}) in the form
\begin{align*}
 c_j^{(r)}(n) &= \sum_{i=0}^j (-1)^{j-i} {j \choose i} \prod_{k=1}^t \frac{(a_k+i+r-1)!}{(i+r-1)!\,a_k!} \\
 &= \left(\prod_{k=1}^t \frac{(a_k+r-1)!}{a_k!\,(r-1)!}\right)\sum_{i=0}^j (-1)^{j-i} {j \choose i} \prod_{k=1}^t \left(\frac{(a_k+r-1+i)!}{(a_k+r-1)!\,i!}\,\frac{i!\,(r-1)!}{(i+r-1)!}\right)
\end{align*}
and apply Lemma \ref{bcrt}.
For the hypergeometric form, we rewrite the last expression as
\begin{align*}
 \frac{c_j^{(r)}(n)}{d_r(n)} &= (-1)^j \sum_{i=0}^j \frac{(-j)^{\overline i}}{i!}\,\prod_{k=1}^t \frac{(a_k +r)^{\overline i}}{r^{\overline i}}
\end{align*}
and note as above that the sum can be extended to an infinite series since
$(-j)^{\overline i} = 0$ if $i > j$.
\end{proof}

\noindent
Specifically for $r=0$, the formula (\ref{eq:thm12}) of Theorem \ref{cjr2} gives
\begin{align*}
\frac{c_j(n)}{d_{j}(n)}&=  \sum_{i=0}^j (-1)^i {j \choose i} \frac{{ j-1  \choose i}^t}{\prod _{k=1}^t {a_k +j -1 \choose i}}
\\
 &=   {}_{t+1}F_t(\{1-j\}_{i=1}^t,-j;\{1-a_i-j \}^{t}_{i=1};1)
\end{align*}
if $n$ has prime factorisation $n = p_1^{a_1}\cdots p_t^{a_t}$.

Clearly these formulae simplify when $n$ is a prime power.
We note that in this case, Lemmata \ref{bcrt} and \ref{binomcd} give
\[
\frac{c_j(p^a)}{d_j(p^a)} = \frac{\binom{a-1}{a-j}}
{\binom{a+j-1}{a}}.
\]
	
\noindent
In the following, we give Dirichlet series for the ratio of divisor functions
$c_j/d_j$
for $j\in\{1, 2, 3\}$,
as well as corresponding Euler products. Note that the term $n=1$ can be
omitted from the Dirichlet series, since $c_j(1) = 0$ for all $j\in\mathbb{N}$.
		 For $j=1$,
		\[ \sum_{n=2}^{\infty} \frac{c_1(n)}{d_1(n)}\frac{1}{n^s} = \sum_{n=2}^{\infty} \frac{1}{n^s} = \zeta(s) -1 .\]
		 For $j=2$, we have $c_2(n) = d_2(n) -2  $ for $n \geq 2 $, so the Dirichlet series for the ratio of divisor functions can be written as
		\[\sum_{n=2}^{\infty} \frac{c_2(n)}{d_2(n)}\,\frac{1}{n^s} = \sum_{n=2}\frac{1}{n^s} - 2 \sum_{n=2}^{\infty} \frac{1}{d_2(n) n^s} = 1 + \zeta(s) - 2 \sum_{n=1}^{\infty} \frac{1}{d_2(n)n^s};  \]
we note the Euler product for the Dirichlet series in the last term,
\begin{align*}
\sum_{n=1}^{\infty} \frac 1{d_2(n)\,n^s}&= \prod_{p \ {\rm prime}} \sum_{j=0}^\infty \frac 1{d_2(p^j)\,p^{sj}}
 = \prod_{p \ {\rm prime}} \sum_{j=0}^\infty \frac 1{(j+1)\,p^{sj}} \\
 &= \prod_{p \ {\rm prime}} \left( p^s \log\left(\frac{1}{1-\frac{1}{p^s}} \right)  \right).
\end{align*}
		 For $j=3 $, we have $c_3(n)= d_3(n) - 3d_2(n) + 3$ for $n \geq 2$, which gives the Dirichlet series for the ratio of divisor functions
		\[ \sum_{n=2}^{\infty} \frac{c_3(n)}{d_3(n)}\,\frac{1}{n^s} = \zeta(s) - 1 - 3 \sum_{n=1}^{\infty} \frac{d_2(n)}{d_3(n)}\frac{1}{n^s} + 3 \sum_{n=1}^{\infty}\frac{1}{d_3(n)}\frac{1}{n^s}; \]
noting
		\[\frac{d_2(p^j)}{d_3(p^j)} = \frac{2}{j+2}, \ \frac{1}{d_3(p^j)} = \frac{2}{(j+1)(j+2)} = \frac{2}{j+1} - \frac{2}{j+2}, \]
we find the Euler products
\begin{align*}
\sum_{n=1}^{\infty} \frac{d_2(n)}{d_3(n)}\frac{1}{n^s} &= \prod_{p \ {\rm prime}} \sum_{j=0}^\infty \frac 2{(j+2)\,p^{sj}}
 = \prod_{p \ {\rm prime}} \sum_{j=2}^\infty \frac{2 p^{2s}}{j\,p^{sj}}
\\\\
 &= \prod_{p \ {\rm prime}} 2p^{2s}\left(   \log\left(\frac{1}{1- \frac{1}{p^s}}\right) - \frac{1}{p^s} \right)
\end{align*}
		and
\begin{align*}
\sum_{n=1}^{ \infty}  \frac{1}{d_3(n)}\frac{1}{n^s} &= \prod_{p \ {\rm prime}} \sum_{j=0}^\infty \left(\frac 2{j+1} - \frac 2{j+2}\right)\frac 1{p^{sj}}
 = \prod_{p \ {\rm prime}} \left(\sum_{j=1}^\infty \frac{2 p^s}{j p^{sj}} - \sum_{j=2}^\infty \frac{2 p^{2s}}{j p^{sj}}\right)
\\
 &= \prod_{p \ {\rm prime}} \left(2 p^s \log \frac 1{1 - \frac 1{p^s}} - 2 p^{2s}\left(\log \frac 1{1-\frac 1{p^s}} - \frac 1{p^s}\right)\right)
\\
 &= \prod_{p \ {\rm prime}} 2 p^{2s} \left(\frac 1{p^s} - \left(1 - \frac 1{p^s}\right)\log\frac 1{1 - \frac 1{p^s}}\right).
\end{align*}

\section{Counting Principal Reversible Squares}
As an illustration for the use of the non-trivial and associated divisor functions,
we show how they can be used to count the different principal reversible
squares of a given size.

A {\it reversible square matrix\/} $M=\left ( M_{i,j}\right )_{i,j \in \mathbb{Z}_n} \in \mathbb{R}^{n\times n}$ is an $n \times n$
matrix with the following symmetry properties (cf. \cite{rOB}, \cite{rSupAlg}),
\begin{description}
\item{(R)} the row and column reversal symmetry
\begin{align*}
 M_{i,j} +M_{i,n+1-j}&=M_{i,k} +M_{i,n+1-k}, \\
 M_{i,j}+M_{n+1-i,j}&=M_{k,j}+M_{n+1-k,j}
 \qquad (i,j,k \in \mathbb{Z}_n),
\end{align*}
\item{(V)} the vertex cross sum property $M_{i,j}+M_{k,l}=M_{i,l}+M_{k,j}$
$(i,j,k,l \in \mathbb{Z}_n)$.
\end{description}
Note that the index calculations are performed in the cyclic ring
$\mathbb{Z}_n := \mathbb{Z}/{n \mathbb{Z}}$, and the top left corner of the matrix has indices
$(1,1) \in \mathbb{Z}_n^2$.

An $n\times n$ {\it principal reversible square\/} is a reversible square matrix $M$ such
that $\{M_{i,j}\mid i,j\in\mathbb{Z}_n\} = \{1,2, \dots, n^2\}$, the entries
in each row and each column appear in increasing order, and
$M_{1,j} = j$ ($j\in\{1,2\}$).

\begin{definition}
Let $n, \alpha\in\mathbb{N}$. The pair of tuples
		\[((i_1, i_2, \dots, i_{\alpha-1}, i_\alpha), \ (j_1, j_2, \ldots, j_{\alpha-1}, j_{\alpha})) \in (\mathbb{N}^\alpha)^2 \]
is called a {\it divisor path set\/} for $n$ (of length $\alpha$) if
		\[i_1 | i_2 | \ldots | i_{\alpha-1}| i_{\alpha}, \qquad 1< i_1<i_2 < \ldots < i_{\alpha-1}<i_{\alpha}= n, \]and
		\[j_1 | j_2 | \ldots | j_{\alpha-1}| j_{\alpha}  | n, \qquad1< j_1 <  j_2 < \ldots <j_{\alpha-1} < j_{\alpha} \leq n. \]
	\end{definition}
	
\begin{theorem}
\label{tdipase}
Let
$n \in \mathbb{N}$. Then from any divisor path set for $n$, a unique $n\times n$ principal
reversible square can be constructed. Conversely, every $n\times n$ principal
reversible square arises from a unique divisor path set.
\end{theorem}
\noindent
For the details of the construction and proof of Theorem \ref{tdipase}, we
refer the reader to Chapter 3 of \cite{rOB}.
In that book, a principal reversible square constructed from a divisor path
set of length $\alpha$ is said to have $\alpha-1$ {\it progressive factors\/}.

Alternatively, Theorem \ref{tdipase} and the
construction can be obtained as a special case, for a two-dimensional square
array, of Theorem 9 of \cite{rHLS}; note that the ratios of consecutive
divisors in the divisor path set correspond to the factors appearing in the
joint ordered factorisation defined in \cite{rHLS}. Specifically, the
above divisor path set corresponds to the joint ordered factorisation
\[
 ((2,j_1), (1,i_1), (2,j_2/j_1), (1,i_2/i_1), (2,j_3/j_2), \dots,
 (2,j_\alpha/j_{\alpha-1}), (1,i_\alpha/i_{\alpha-1}), (2, n/j_\alpha))
\]
of $(n, n)$,
with the last entry omitted if
$j_\alpha = n.$

Using the bijection between divisor path sets and principal reversible squares
given by Theorem \ref{tdipase}, we can count the number of different principal
reversible squares of size $n\times n$ in terms of the non-trivial and associated
divisor functions of $n$ as follows.

\begin{theorem}\label{mcl6}
Let
$n\in\mathbb{N}$.
The number of different $n\times n$ principal reversible squares
is given by
			\[N_n=\sum_{j=1}^{\Omega(n)}c_j(n)\left (c_j(n)+c_{j+1}(n)\right )=\sum_{j=1}^{\Omega(n)}c_j^{(0)}(n)c_j^{(1)}(n).  \]
\end{theorem}
		\begin{proof}
By Theorem \ref{tdipase}, it is sufficient to count the number of different
divisor path sets for $n$.

Suppose
$((i_1, \dots, i_\alpha), (j_1, \dots, j_\alpha))$
is a divisor path set for $n$ of length $\alpha$. Then the left-hand tuple gives
an ordered factorisation of $n$ into $\alpha$ factors,
\[
 i_1 \ \frac{i_2}{i_1} \ \frac{i_3}{i_2} \ \cdots \ \frac{i_\alpha}{i_{\alpha-1}} = n
\]
with all factors $> 1$; there are
$c_\alpha(n)$
such non-trivial factorisations.

The right-hand tuple gives an ordered factorisation of $n$ into $\alpha+1$
factors,
\[
 j_1 \ \frac{j_2}{j_1} \ \frac{j_3}{j_2} \ \cdots \ \frac{j_\alpha}{j_{\alpha-1}} \ \frac{n}{j_\alpha} = n,
\]
where the last factor is $> 1$ if and only if $j_\alpha < n$, and all other
factors are $> 1$.
%
As the two cases are
mutually exclusive, the total number of different factorisations is then given by
$c_{\alpha}(n) + c_{\alpha+1}(n) = c_{\alpha}^{(1)}(n)$, by Lemma \ref{mixcjr}.
The statement of the theorem follows by summing over $\alpha\in\mathbb{N}$,
noting that $c_\alpha(n) = 0$ if $\alpha > \Omega(n)$.
\end{proof}

\begin{remark}
Combining the formula of Theorem \ref{mcl6} with Lemma \ref{mixcjr}, the count
$N_n$ can be expressed in terms of the (multiplicative) divisor functions,
\begin{equation}
\label{sumdiprob}
N_n=\sum_{j=1}^{\Omega (n)} \sum_{l=1}^j \sum_{m=0}^{j} (-1)^{l+m} \binom{j}{l}
   \binom{j}{m} d_l(n) d_{m+1}(n).
\end{equation}
Using the prime factorisation
$n = \prod\limits_{k=1}^t p_k^{a_k}$, this takes the form
\[
N_n=\sum_{j=1}^{\Omega (n)} \sum_{l=1}^j \sum_{m=0}^{j} (-1)^{l+m} \binom{j}{l}   \binom{j}{m} \prod_{k=1}^t {a_k +l -1 \choose l-1} {a_k +m \choose m}.
\]
We note that the terms of the sum in (\ref{sumdiprob}) bear some similarity
to the sum underlying the additive divisor problem (\ref{adp}).
\end{remark}

\begin{corollary}
\label{tc1}
Let $n\in\mathbb{N}$. Then $N_n = 1$ if and only if $n$ is prime.
\end{corollary}
\begin{proof}
If $n$ is prime, then $c_1(n) = 1$ and $c_j(n) = 0$ for all $j\ge 2$, and it follows that
$N_{n}=c_1(n)\left(c_1(n)+c_{2}(n)\right )=1$.
Conversely, suppose
 $n \ge 2$ is an integer such that
$N_n = 1$. By Theorem \ref{mcl6},
\[
 N_n = \sum_{j=1}^{\Omega(n)} c_j(n)^2 + \sum_{j=1}^{\Omega(n)} c_j(n) c_{j+1}(n).
\]
As all terms of these sums are non-negative and $c_1(n) = 1$, the total can
equal 1 only if $c_j(n) = 0$ $(j \ge 2)$, which implies that
$n$ is prime.
\end{proof}

\begin{corollary}\label{tc2}
Let $n = p^a$ with $a \in\mathbb{N}$ and prime $p$.
Then
\[
N_n=\binom{2a-1}{a}.
\]
\end{corollary}
\begin{proof}
Theorem \ref{mcl6} and Lemma \ref{binomcd} give
\[
N_{p^a}=\sum_{j=1}^a \binom{a-1}{j-1}\binom{a}{j}=\sum_{j=0}^{a-1} \binom{a-1}{j}\binom{a}{j+1} = \binom{2a-1}a
\]
by combinatorial identity (3.20) of \cite{rHWG}.
\end{proof}

Counting principal reversible squares is of interest not only in view of their
bijection to most perfect squares \cite{rOB}, but also because of their
relationship with sum-and-distance systems. In the present context, these are
composed of two finite component sets, of equal cardinality, of natural
numbers, such that the numbers formed by considering all sums and all
absolute differences of all pairs of numbers, each taken from one of the
component sets, with or without inclusion of the component sets themselves, combine to an arithmetic progression without repetitions.
Such systems arise naturally from the question of constructing a certain type
of rank 2 traditional magic squares using the formulae given in \cite{rSupAlg}.
We refer the reader to \cite{rHLS} for further details and for the extension
of the following definitions and of Theorem \ref{sdsisprv} to any finite
number of component sets of arbitrary finite cardinality.
	
\begin{definition}
Let $m\in\mathbb{N}$ and consider positive integers $a_j, b_j \in\mathbb{N}$
$(j\in\{1,\dots,m\})$ such that
\[
a_1 < a_2 < \cdots < a_m, \qquad b_1 < b_2 < \cdots < b_m.
\]
We call $\{\{a_j : j\in\{1,\dots,m\}\}, \{b_j : j\in\{1,\dots,m\}\}\}$
\hfill\break
a) an $m+m$ {\it (non-inclusive) sum-and-distance system\/} if
\[
 \{a_j+b_k, |a_j-b_k| : j,k\in\{1,\dots,m\}\} = \{1, 3, 5, \dots, 4m^2-1\};
\]
b) an $m+m$ {\it inclusive sum-and-distance system\/} if
\[
 \{a_j, b_k, a_j+b_k, |a_j-b_k| : j,k\in\{1,\dots,m\}\} = \{1, 2, 3, \dots, 2m(m+1)\}.
\]
\end{definition}

It is easy to see that the conditions in a) and b) above are equivalent to
\begin{align}
\label{sudi2}
 \{\pm a_j : j\in\{1,\dots, m\}\} &+ \{\pm b_k : k\in\{1,\dots,m\}\}\nonumber
\\
 &= \{-4m^2+1, -4m^2+3, \dots, 4m^2-3, 4m^2-1\}
\end{align}
and
\begin{align}
\label{sudi3}
 \{0, \pm a_j : j\in\{1,\dots,m\}\} &+ \{0, \pm b_k : k\in\{1,\dots,m\}\}\nonumber
\\
 &= \{-2m(m+1), -2m(m+1)+1, \dots, 2m(m+1)\},
\end{align}
respectively,
where we use the usual sum of sets
$A + B = \{x + y : x\in A, y\in B\}$.

	\begin{example}
%
For $m=3$ there are the seven $3+3$ (non-inclusive) sum-and-distance systems
\begin{align*}
&\{\{ 1 , 3 , 5 \},\{ 6 , 18 , 30\}\},\quad\{\{ 1 , 7 , 9\}, \{2 , 22 , 26\}\},\quad\{\{1 , 11 , 13\}, \{14 , 18 , 22\}\},
		\\
&\{\{1 , 23 , 25\}, \{2 , 6 , 10\}\},\quad\{\{3 , 9 , 15\}, \{16 , 18 , 20\}\},\quad\{\{3 , 21 , 27\}, \{4 , 6 , 8\}\},
\\
&\{\{7 , 9 , 11\}, \{12 , 18 , 24\}\},
\end{align*}
		but just the one inclusive $3+3$ sum-and-distance system $\{\{1 , 2 , 3\}, \{7 , 14 , 21\}\}$.
\end{example}

Sum-and-distance systems of the non-inclusive and inclusive variety are
intimately connected with principal reversible squares.
Indeed, let $n\in\mathbb{N}$ and consider a principal reversible square
$(M_{j,k})_{j,k\in\mathbb{Z}_n}$; for all $j\in\{1,\dots,n\}$, define
$\alpha_j = M_{1,j} - 1$ and ${\beta_j = M_{j,1} - 1}$.
From property (V) of the reversible square and the fact that $M_{1,1}=1$, it then follows that
\begin{equation}\label{sudi4}
 M_{j,k} = M_{j,1} + M_{1,k} - 1 = \alpha_j + \beta_k + 1 \qquad (j,k\in\{1,\dots,n\});
\end{equation}
note that $\alpha_1 = \beta_1 = 0$.
By property (R),
\begin{equation}\label{sudi1}
 \alpha_j + \alpha_{n+1-j} = \alpha_n, \quad \beta_j + \beta_{n+1-j} = \beta_n
 \qquad (j\in\{1, \dots, n\}).
\end{equation}
Now suppose $n$ is even, $n = 2m$. Then setting
\[
 a_j = \alpha_{m+j} - \alpha_{m+1-j}, \quad b_j = \beta_{m+j} - \beta_{m+1-j}
 \qquad (j\in\{1, \dots, m\})
\]
defines an $m+m$ non-inclusive sum-and-distance system.
Indeed, it follows from (\ref{sudi1}) that
\[
 a_j = 2 \alpha_{m+j} - \alpha_n, \quad -a_j = 2 \alpha_{m+1-j} - \alpha_n
 \qquad (j\in\{1,\dots,m\})
\]
and correspondingly for $\pm b_k$ $(k\in\{1,\dots,m\})$, so
the equality (\ref{sudi2}) can be verified using the facts that
$\alpha_n + \beta_n = M_{n,n}-1 = n^2-1$ and that
\[
 \{\alpha_j + \beta_k : j,k\in\{1,\dots,n\}\} = \{M_{j,k}-1 : j,k\in\{1,\dots,n\}\} = \{0, 1, \dots, n^2-1\}.
\]
Conversely, given an $m+m$ non-inclusive sum-and-distance system and setting
\begin{align*}
 \alpha_{m+j} = \frac{a_m + a_j}2, &\quad \alpha_{m+1-j} = \frac{a_m - a_j}2 \qquad (j\in\{1,\dots,m\}),
\\
 \beta_{m+k} = \frac{b_m + b_k}2, &\quad \beta_{m+1-k} = \frac{b_m - b_k}2 \qquad (k\in\{1,\dots,m\}),
\end{align*}
equation (\ref{sudi4}) defines a principal reversible square (or its transpose).

Similarly, if $n = 2m+1$ is odd, then an $m+m$ inclusive sum-and-distance
system can be obtained from any $n\times n$ principal reversible square by
setting
\[
 a_j = \frac{\alpha_{m+1+j}-\alpha_{m+1-j}}2, \quad
 b_j = \frac{\beta_{m+1+j}-\beta_{m+1-j}}2 \qquad (j\in\{1,\dots,m\});
\]
now by (\ref{sudi1}),
\[
 a_j = \alpha_{m+1+j} - \frac{\alpha_n}2, \quad
 -a_j = \alpha_{m+1-j} - \frac{\alpha_n}2 \qquad (j\in\{1,\dots,m\}),
\]
and similarly for $b_k$ $(k\in\{1,\dots,m\})$, and hence equality (\ref{sudi3})
follows in analogy to the above.
Conversely, setting
\begin{align*}
 \alpha_{m+1+j} = a_m + a_j, \quad \alpha_{m+1-j} = a_m - a_j \qquad (j\in\{1,\dots, m\}),
\\
 \beta_{m+1+k} = b_m + b_k, \quad \beta_{m+1-k} = b_m - b_k \qquad (k\in\{1,\dots, m\}),
\end{align*}
and
$\alpha_{m+1} = a_m$, $\beta_{m+1} = b_m$, we obtain a principal reversible
square (or its transpose), via (\ref{sudi4}), from an inclusive sum-and-distance
system.

Thus we have proven the following statement.
\begin{theorem}
\label{sdsisprv}
Let $m\in\mathbb{N}$. Then there is a bijection between the $m+m$ non-inclusive sum-and-distance systems
and the $2m \times 2m$ principal reversible squares, and there is a bijection
between the $m+m$ inclusive sum-and-distance systems and the $(2m+1)\times (2m+1)$ principal reversible squares.
\end{theorem}

In conjunction with Theorem \ref{mcl6}, this gives the following counting of
non-inclusive and inclusive sum-and-distance systems.

\begin{corollary}
Let $m\in\mathbb{N}$. Then there are
\[
 N_{2m} = \sum_{j=1}^{\Omega(2m)} c_j^{(0)}(2m)\,c_j^{(1)}(2m)
\]
different $m+m$ non-inclusive sum-and-distance systems, and
\[
 N_{2m+1} = \sum_{j=1}^{\Omega(2m+1)} c_j^{(0)}(2m+1)\,c_j^{(1)}(2m+1)
\]
different $m+m$ inclusive sum-and-distance systems.
\end{corollary}

To conclude, we briefly note that $m+m$ sum-and-distance systems of either variety have the general property that the sum of squares of all entries of
their component sets is invariant, determined only by the size $m$.
\begin{theorem}
Let $m\in\mathbb{N}$ and $\{\{a_1,\dots,a_m\},\{b_1,\dots,b_m\}\}$ a
(non-inclusive or inclusive) sum-and-distance system. Then
\[
\sum_{j=1}^m (a_j^2 + b_j^2) = \left\{\begin{matrix}
 {\displaystyle \frac 1{3!}}\,(2m)((2m)^4-1) & \textit{in the non-inclusive case,} \\
 {\displaystyle \frac 1{4!}}\,(2m+1)((2m+1)^4-1) & \textit{in the inclusive case.}
\end{matrix}\right.
\]
\end{theorem}
\begin{proof}
In the non-inclusive case we use the formula
\[
 \sum_{j=1}^n (2j-1)^2 = \frac{n (4n^2-1)}3 \qquad (n\in\mathbb{N})
\]
to find
\begin{align*}
 2m \sum_{j=1}^m (a_j^2 + b_j^2) &= \sum_{j=1}^m \sum_{k=1}^m ((a_j+b_k)^2 + (a_j-b_k)^2)
\\
 &= \sum_{j=1}^{2m^2} (2j-1)^2
 = \frac{1}{6}\, 4m^2(16m^4-1).
\end{align*}
In the inclusive case, we similarly use the formula
\begin{align*}
\sum_{j=1}^n j^2 &= \frac{n(n+1)(2n+1)}6 = \frac{(2n+1-1)(2n+1+1)(2n+1)}{24}
\\
 &= \frac{((2n+1)^2-1)(2n+1)}{24}
 \qquad (n\in\mathbb{N})
\end{align*}
and the identity for $n=2m(m+1)$
\[
 2n+1 = 4m^2+4m+1 = (2m+1)^2
\]
to calculate
\begin{align*}
 (2m+1) \sum_{j=1}^m (a_j^2+b_j^2) &= \sum_{j=1}^m \sum{k=1}^m ((a_j+b_j)^2 + (a_j-b_j)^2) + \sum_{j=1}^m a_j^2 + \sum_{j=1}^m b_j^2
\\
 &=\sum_{j=1}^{2m(m+1)} j^2
 = \frac 1{24}\,((2m+1)^4-1)\,(2m+1)^2.
\end{align*}
\end{proof}




\end{document}